\documentclass[11pt,reqno]{amsart}
\usepackage[utf8]{inputenc}  
\usepackage[T1]{fontenc}
\usepackage{graphicx}
\usepackage{cite}
\usepackage{amssymb, amsmath}
\usepackage[normalem]{ulem}

\usepackage{fixltx2e}
\usepackage{newunicodechar}
\usepackage{color,xcolor}
\usepackage[normalem]{ulem}

\usepackage{hyperref}

\newtheorem{thm}{Theorem}[section]
\newtheorem{lem}[thm]{Lemma}
\newtheorem{cor}[thm]{Corollary}
\newtheorem{prop}[thm]{Proposition}

\theoremstyle{definition}

\theoremstyle{remark} 
\newtheorem{rem}[thm]{Remark}
\newtheorem{rems}[thm]{Remarks}

\numberwithin{equation}{section}

\DeclareMathAlphabet{\mathpzc}{OT1}{pzc}{m}{it}

\newcommand{\N}{\mathbb N}
\newcommand{\R}{\mathbb R}
\newcommand{\C}{\mathbb C}
\newcommand{\K}{\mathbb K}

\newcommand{\kL}{\mathcal L}
\newcommand{\ve}{\varepsilon}

\newcommand{\rd}{\mathrm{d}}
\newcommand{\p}{\partial}

\newcommand{\compact}{\mbox{\mbox{\raisebox{0.11cm}{$\scriptscriptstyle\subset
\hspace{-.76mm}\subset$}}\hspace{-1.7mm}$\longrightarrow$}}

\begin{document}

 \author[B.--V. Matioc]{Bogdan--Vasile Matioc}
\address{Fakult\"at f\"ur Mathematik, Universit\"at Regensburg,   93053 Regensburg, Germany.}
\email{bogdan.matioc@ur.de}

\author[Ch. Walker]{Christoph Walker}
\address{Leibniz Universit\"at Hannover,
Institut f\"ur Angewandte Mathematik,
Welfengarten 1,
30167 Hannover,
Germany.}
\email{walker@ifam.uni-hannover.de}

\title[Semilinear parabolic evolution problems with  H\"older continuous nonlinearities]{Global solutions for semilinear parabolic evolution problems with H\"older continuous nonlinearities}

\begin{abstract}
It is shown that semilinear parabolic evolution equations $u'=A+f(t,u)$  featuring  H\"older continuous nonlinearities $ f=f(t,u)$ with  at most linear growth
possess global strong solutions for a general class of initial data.  
The abstract results are applied  to a  recent  model describing front propagation in   bushfires  and   in the context of  a reaction-diffusion system.
\end{abstract}


\maketitle
\section{Introduction}\label{Sec:1}

For semilinear parabolic problems
\begin{equation}\label{EE}
u'=Au+f(t,u)\,,\quad t\in(0,T]\,,\qquad u(0)=u^0\,,
\end{equation}
involving a generator $A$ of an analytic semigroup on a Banach space $E_0$ and a locally  Lipschitz continuous nonlinearity $f$, there is a well-established theory 
for well-posedness  based on Banach's fixed point theorem, e.g. see \cite[Section~12]{Amann_Teubner} and the references therein.
 The situation is different if~$f=f(t,u)$ is not locally Lipschitz  with respect to $u$ and less seems to be known. 
 If compactness properties are available, (versions of) Schauder's fixed point theorem may be used to derive existence 
 -- but not uniqueness  -- results, see e.g. the textbooks~\cite{LSU,Lieberman} or also \cite{Arendt-Chill}. 
 We also refer to \cite{MeyerNeedham} for the treatment of non-Lipschitz semilinear parabolic scalar equations by comparison principle arguments.

The aim of this paper is to establish the existence of strong solutions to~\eqref{EE} 
in the context of nonlinearities~${f=f(t,u)}$ with at most linear growth in $u$,  which are not necessarily  locally Lipschitz continuous with respect to the variable  $u$,
under  optimal (i.e. low) regularity assumptions on the initial value. 
 
To be more precise,  in the following let $E_0$ and $E_1$ be two Banach spaces  over $\K\in \{\R,\C\}$ with compact, continuous, and dense embedding 
$$
E_1\,  \stackrel{d}{\compact}\, E_0\,. 
$$
 For each $\theta\in (0,1)$, let $(\cdot,\cdot)_\theta$ be an arbitrary admissible interpolation functor of 
 exponent $\theta$  (see \cite[I.~Section~2.11]{LQPP}) and denote by ${E_\theta:= (E_0,E_1)_\theta}$ the corresponding  interpolation space with norm~\mbox{$\|\cdot\|_\theta$}.
 Then, the embeddings
$$
E_1\,  \stackrel{d}{\compact}\, E_\theta\, \stackrel{d}{\compact}\, E_\vartheta \,  \stackrel{d}{\compact}\, E_0
\,,\qquad 0< \vartheta< \theta< 1\,,
$$
are all  dense and compact.
 We fix an infinitesimal generator $A:E_1\to E_0$  of a strongly continuous analytic semigroup~${\{e^{tA}\,:\, t\geq0\} }\subset{\kL(E_0)}$  
 and assume that $f:[0,T]\times E_\xi\to E_\gamma$ for some~${\gamma,\, \xi\in[0,1) }$ and~${T\in(0,\infty)}$.

 We first point out the following  result guaranteeing the global existence
  of (at least) a global strong solution to \eqref{EE} for initial data $u_0\in E_\alpha$ for which 
$f(0,u_0)$ is defined,  that is, for $\alpha\in[ \xi,1)$, under the general assumption that $f(t,u)$  grows at most linearly in $u$  and  is continuous 
(along with an additional H\"older continuity property in the  limiting case $\gamma=0$).

 \begin{prop}\label{MT0}
Let $\gamma,\,\xi \in [0,1)$  and assume that $f\in C\big([0,T]\times E_\xi,E_\gamma\big)$ is such that there is a constant $C>0$ with
\begin{equation}\label{w1}
\|f(t,u)\|_\gamma\le C\big(1+\|u\|_\xi\big)\,,\qquad t\in[0,T]\,,\quad  u\in E_\xi\,.
\end{equation}
If $\gamma=0$  assume additionally that there exists $\vartheta_0\in(0,1)$ such that for each $R>0$ there is a constant~${K(R)>0}$  with 
\begin{align}\label{w2MT0}
\|f(t,u)-f(s,v)\|_0\le K(R)\big(\vert t-s\vert^{\vartheta_0}+\|u-v\|_\xi^{\vartheta_0}\big),\qquad t,s\in [0,T]\,,\quad u,v\in \bar{\mathbb{B}}_{E_\xi}(0,R)\,.
\end{align}
Then, if $u^0\in E_\alpha $ for some $\alpha\in[\xi,1)$, the Cauchy problem~\eqref{EE} admits a solution 
\begin{equation*}
 u\in C^1\big((0,T],E_0\big)\cap C\big((0,T],E_1\big)\cap  C \big([0,T],E_\alpha\big)\,.
\end{equation*}
\end{prop}

The main purpose of this paper is to extend the existence theory for the evolution problem~\eqref{EE} by  providing   a global strong solution to~\eqref{EE}
 for a general initial datum~$u^0$  in the ambient space ~$E_0$. 
 To this end we require, additionally to being  linearly bounded in $u$, that $f$ is locally H\"older continuous with respect to $u$.  
 More precisely, we  fix
 \begin{subequations}\label{ASS}
\begin{equation}\label{a1A}
0\leq \gamma<1\,,\qquad  0< \xi <\min\Big\{1,\frac{1}{q}\Big\}\,, \qquad 0<\vartheta_j\leq  q_j\,,\quad 1\leq j\leq   m\in\N^*\,,
\end{equation} 
where 
\[
q:=\max\{q_j\,:\, 1\leq j\leq m\}>0\,,
\]
and  assume that  $f\in C\big( (0,T]\times E_\xi, E_\gamma\big)$  is such that there exists a positive constant $C>0$ with
\begin{align}
&\|f(t,u)\|_\gamma\le C\big(1+\|u\|_\xi\big)\label{a1B1}
\end{align}
and
\begin{align}
&\|f(t,u)-f(t,v)\|_{\gamma}\le  C\sum_{j=1}^  m\big(1+\|u\|_{\xi}^{q_j-\vartheta_j}+\|v\|_{\xi}^{q_j-\vartheta_j}\big)\|u-v\|_{\xi}^{\vartheta_j}\label{a1B2}
\end{align}
\end{subequations}
 for all~${t\in (0,T]}$ and $u, v\in E_\xi$.
 In the particular case $\gamma=0$ we additionally require (similarly as in Proposition~\ref{MT0}, but in a slightly weaker form) that there exists~$\vartheta_0\in(0,1)$  such that for each~${R>0}$ there is a constant $K(R)>0$ with
\begin{equation}\label{a1B'}
\|f(t,u)-f(s,u)\|_{0}\le  K(R)\,|t-s|^{\vartheta_0},\qquad t, s\in (0,T]\,, \quad u\in  E_\xi\cap \bar{\mathbb{B}}_{E_0}(0,R)\,.
\end{equation}

 Our main result then reads:

\begin{thm}\label{MT1}
Suppose \eqref{ASS} and if $\gamma=0$ also assume \eqref{a1B'}.
Then,  given  $u^0\in E_0$, 
the Cauchy problem \eqref{EE} possesses a  global strong  solution
\begin{equation*} 
\begin{aligned}
 u&\in C^1\big((0,T],E_0\big)\cap C\big((0,T],E_1\big)\cap  C \big([0,T],E_0\big)\,.
\end{aligned}
\end{equation*}
Furthermore, if $u^0\in E_\alpha$ for some $\alpha\in [0,\min\{1,1+\gamma-\xi\})$, 
then  $C \big([0,T],E_\alpha\big)$.
\end{thm}

Theorem~\ref{MT1} provides strong solutions assuming minimal regularity  on the initial value. 
In particular,  $f(0,u^0)$ needs not be defined under the assumptions of Theorem~\ref{MT1} (in contrast to Proposition~\ref{MT0}).
For this reason, if $u^0\in E_\alpha\setminus E_\xi$, the  continuity  in $E_\alpha$ of the solution at $t=0$ is guaranteed only for  possibly restricted  $\alpha$.
 
 \begin{rems} 
{\bf (i)} If~$1\leq \vartheta_j\leq q_j$ for all  $1\leq j\leq  m $, then $f=f(t,u)$ is locally Lipschitz continuous in~$u$ and  the classical theory on well-posedness for semilinear parabolic equations applies, see  e.g.  \cite{Amann_Teubner} for a general treatment (or \cite{MW24} for a similar approach to the one chosen herein requiring minimal regularity assumptions on the initial value).\vspace{1mm}

{\bf (ii)}   In general, however, the nonlinearity $f=f(t,u)$ in~\eqref{EE} is not locally Lipschitz continuous in~$u$ as  
the exponents $\vartheta_j$ in~\eqref{a1B2} are allowed to be less than~$1$. 
Thus, our approach to Proposition~\ref{MT0} and Theorem~\ref{MT1},   relying on the Leray-Schauder fixed point theorem, does not provide  uniqueness of solutions.\vspace{1mm}

{\bf (iii)}  The linear bound \eqref{a1B1}  is not in contradiction with assumption~\eqref{a1B2} (even if $q$  is possibly larger than $1$).
  This is well exemplified in  Section~\ref{Sec:3} by application  \eqref{Bush} with $\nu=1$.
\end{rems}
 
Combining  Proposition~\ref{MT0} and Theorem~\ref{MT1}, we may establish  a global existence result for~${T=\infty}$. 

\begin{cor}\label{MT2}
  Let  $\gamma,\,\xi \in [0,1)$,  $f\in C\big( [0,\infty)\times E_\xi, E_\gamma\big)$, and assume that for each $T>0$ there exists a  constant~$C(T)>0$ with 
 \begin{align}\label{linbound}
\|f(t,u)\|_{\gamma}\le   c(T) (1+\|u\|_{\xi})\,,\qquad t\in [0,T]\,,\quad u\in E_\xi\,.
\end{align}

{\bf (a)} Let $\gamma>0$. Then \eqref{EE} has a global strong solution
\begin{equation*} 
 u\in C^1\big((0,\infty),E_0\big)\cap C\big((0,\infty),E_1\big)\cap  C \big([0,\infty),E_\alpha\big)
\end{equation*}
provided either $u^0\in E_\alpha$ with $\alpha\in [\xi,1)$, or $f$ satisfies \eqref{ASS} for some $T=T_0>0$ and $u^0\in E_\alpha$ with~$\alpha\in [0,\min\{1,1+\gamma-\xi\})$.\vspace{1mm}

{\bf (b)} Let $\gamma=0$ and suppose
 there is $\vartheta_0\in (0,1)$ and for $T,R>0$ there is $K=K(T,R)>0$ such that
\begin{align}\label{w2}
\|f(t,u)-f(s,v)\|_0\le K \big(\vert t-s\vert^{\vartheta_0}+\|u-v\|_\xi^{\vartheta_0}\big)\,,\qquad t,s\in [0,T]\,,\quad u,v\in \bar{\mathbb{B}}_{E_\xi}(0,R)\,.
\end{align}
Then \eqref{EE} has a global strong solution
\begin{equation*} 
 u\in C^1\big((0,\infty),E_0\big)\cap C\big((0,\infty),E_1\big)\cap  C \big([0,\infty),E_\alpha\big)
\end{equation*}
provided  either $u^0\in E_\alpha$ with $\alpha\in [\xi,1)$, or 
$f$ satisfies \eqref{ASS}  and \eqref{a1B'} for some $T=T_0>0$ and~$u^0\in E_\alpha$ with~$\alpha\in [0,1-\xi)$.
\end{cor}

  The proofs of  Proposition~\ref{MT0}, Theorem~\ref{MT1}, and Corollary~\ref{MT2} are presented in Section~\ref{Sec:2}. 
 In  Section~\ref{Sec:3} we  apply the abstract results to models describing front propagation in bushfires  and to reaction-diffusion models.

\section{Proof of the main results}\label{Sec:2}

  We provide first the details of the proof Theorem~\ref{MT1} (for which  the arguments are a bit more delicate)
  and subsequently only sketch the proofs of  Proposition~\ref{MT0}  and Corollary~\ref{MT2}.\\

Since $A$ is the generator of the strongly continuous analytic semigroup $\{e^{tA}\,:\, t\geq0\}$,  the  Cauchy problem \eqref{EE} can be formulated (under suitable assumptions on $u$ and $f$, see e.g.~\cite[II.~Remarks~2.1.2~(a)]{LQPP}) as the  fixed point equation 
\begin{equation}\label{ttp}
u(t)=e^{tA}u^0+ \int_0^t e^{(t-\tau)A} f(\tau,u(\tau))\,\rd \tau\,,\quad t\in[0,T]\,,
\end{equation}
which we shall solve  using the  Leray-Schauder fixed point theorem.

\subsection{Proof of Theorem \ref{MT1}}

Assuming the premises of Theorem \ref{MT1}, we set
\[
\vartheta:=\min\{\vartheta_j\,:\, 1\leq j\leq  m \}>0
\] 
and choose $\mu\in\R$ such that
\begin{equation}\label{much}
 \xi<\mu<\min\Big\{1,\frac{1}{q}\Big\}\,.
\end{equation}
We denote by $$X_T:=C_{\mu}((0,T],E_\xi)$$ the Banach space consisting of all functions~${u\in C((0,T],E_\xi)}$ such that 
$t^{\mu} u(t)\rightarrow 0$ in $E_\xi$ as~$t\rightarrow 0$, which is equipped with the norm
\begin{equation*}
  \|u\|_{X_T}:= \sup\big\{ t^{\mu}\, \|u(t)\|_{\xi} \,:\, t\in (0,T]\big\}\,.
\end{equation*}
Our goal is to prove that the function~${F:X_T\to X_T}$,
defined by 
\begin{equation}\label{fpe}
F(u)(t):=e^{tA}u^0+ \int_0^t e^{(t-\tau)A} f(\tau,u(\tau))\,\rd \tau\,,\qquad t\in[0,T]\,,\quad u\in X_T\,,
\end{equation}
satisfies the hypotheses of the Leray-Schauder fixed point theorem  \cite[Theorem 11.3]{Gilbarg2001} and thus admits a fixed point in $X_T$,  which is  then shown to be a global strong solution to \eqref{EE}.\medskip
 
\noindent{\bf Well-definedness.} In order to establish that $F(u)\in X_T$ for each $u\in X_T$, we first recall  from~\cite[II.~Lemma~5.1.3]{LQPP} the estimates 
\begin{equation}\label{e6}
\|e^{tA}\|_{\mathcal{L}(E_\theta)}+  t^{\alpha-\beta_0}\,\|e^{tA}\|_{\mathcal{L}(E_\beta,E_\alpha)} 
 \le M
\,, \qquad 0\le  t\le T\,, 
\end{equation}
 for $\theta\in[0,1]$ and $0\le \beta_0\le \beta\leq \alpha\le 1$ 
 with $\beta_0<\beta$ if $0<\beta<\alpha<1$,  where $M=M(T)$ depends also on these parameters.

Given $u,\, v\in X_T$ with $\|u\|_{X_T}\leq L$ and $\|v\|_{X_T}\leq L$ (for some arbitrary $L>0$),  assumption~\eqref{a1B2} implies 
\begin{align}\label{e9}
\|f(t,u(t))-f(t,v(t))\|_{\gamma}&\le c(L)\sum_{j=1}^m  \big( 1+t^{-\mu (q_j-\vartheta_j)}\big) \|u(t)-v(t)\|_{\xi}^{\vartheta_j}\nonumber\\
&\leq c (T,L)\, t^{-\mu q }\, \|u-v\|_{X_T}^{\vartheta}\,,\qquad t\in (0,T]\,,
\end{align}
 while \eqref{a1B1} yields
\begin{align}\label{e8}
\|f(t,u(t))\|_{\gamma}\le   c(T,L)\, t^{-\mu }\,,\quad t\in (0,T]\,.
\end{align}

Fix $u\in X_T$ with $\|u\|_{X_T}\leq L$ (for some arbitrary $L$). 
Let   $\eta\in(0,1)$ be given and define $\gamma_0=\gamma_0(\eta)$ as 
\begin{equation}\label{gamma_0}
\text{$\gamma_0:=0$\, if\, $\gamma=0$ \qquad and\qquad $\gamma_0\in(0,\min\{\gamma,\, \eta\})$\, if\, $\gamma>0$\,.}
\end{equation}
In view of \eqref{e6} and \eqref{e8}  we obtain
\begin{align}
\|F(u)(t)\|_\eta &\le \|e^{tA}\|_{\mathcal{L}(E_0,E_\eta)}\, \|u^0\|_{0}
+\int_0^t \|e^{(t-\tau)A}\|_{\mathcal{L}(E_\gamma,E_\eta)} \,\|f(\tau,u(\tau))\|_\gamma\,\rd \tau \nonumber\\
&\le   c t^{-\eta}\, \|u^0\|_{0} +  c(T,L) \int_0^t (t-\tau)^{\gamma_0-\eta}\tau^{-\mu }\,\rd \tau \nonumber\\
&\le   c(T,L) \big( t^{-\eta}+ t^{1+\gamma_0-\eta -\mu }{\sf B}(1+\gamma_0-\eta,1-\mu )\big)\label{eqeta}
\end{align}
for $t\in (0,T]$,  where ${\sf B}$ denotes the Beta function.
Therefore, for $\eta=\xi$,  we obtain
\begin{align}\label{o1}
t^\mu \|F(u)(t)\|_\xi &\le     c(T,L)\big( t^{\mu-\xi} +  t^{1+\gamma_0-\xi}\big)\,,\quad t\in (0,T]\,,
\end{align}
and the right hand side converges to zero as $t\to0$ since  $\xi<\mu <1$.

Given $\ve\in (0,T)$, set $u_\ve(t):=u(t+\ve)$ for $t\in [0,T-\ve]$. 
Then, $u_\ve\in C([0,T-\ve],E_\xi)$ and the continuity of $f$ implies $f(\cdot+\ve,u_\ve)\in C([0,T-\ve],E_\gamma)$.
 Moreover, the definition of~$F(u)$  entails
\begin{equation}\label{P1d}
F(u)(t+\ve)=e^{tA}F(u)(\ve)+\int_0^t e^{(t-\tau)A} f(\tau+\ve,u_\ve(\tau))\,\rd \tau\,,\quad t\in [0,T-\ve]\,.
\end{equation}
 Hence, \cite[II.Theorem~5.3.1]{LQPP}  yields
\begin{equation*}  
F(u)(\cdot+\ve)\in C\big((0,T-\ve],E_\theta\big)\,,\qquad \theta\in (0,1)\,,\quad \ve\in(0,T)\,,
\end{equation*}
that is,
\begin{equation} \label{Xregdes2}
F(u)\in C\big((0,T],E_\theta\big)\,,\quad \theta\in (0,1)\,.
\end{equation}
Together with \eqref{o1}  we conclude $F(u)\in X_T$.\medskip

\noindent{\bf  Continuity.}
 Given $u,\,v\in X_T$  with  $\|u\|_{X_T}\leq L $ and  $\|v\|_{X_T}\le L$, we obtain in view of \eqref{e6} and~\eqref{e9}, similarly as above,  that
\begin{align}
\|F(u)(t)-F(v)(t)\|_\xi&\le\int_0^t \|e^{(t-\tau)A}\|_{\mathcal{L}(E_\gamma,E_\xi)} \,\|f(\tau, u(\tau))-f(\tau, v(\tau))\|_\gamma\,\rd \tau\nonumber\\
&\le  c\|u-v\|_{X_T}^\vartheta {\sf B}(1+\gamma_0-\xi, 1-\mu q)t^{1+\gamma_0-\xi  - \mu q },\qquad   t\in(0,T],\label{c111}
\end{align}
with $\gamma_0=\gamma_0(\xi)$ as in \eqref{gamma_0}.
Since $\mu q<1$ and $\xi<\mu<1$, we get 
$$
\|F(u)-F(v)\|_{X_T}\leq c_1(T,L)\,\|u-v\|_{X_T}^\vartheta\,.
$$ 
This proves in particular the (H\"older) continuity of $F$.
\medskip

\noindent{\bf Compactness.} We show that $F:X_T\to X_T$ is compact. 
Let therefore $(u_k)_k$ be a bounded sequence in $X_T$ and choose $L>0$ such that $\|u_k\|_{X_T}\leq L$
for all $k\in\N$. We prove that $(F(u_k))_k$ is relatively compact in $X_T$ by applying the 
Arzel\`a-Ascoli theorem to the sequence $(v_k)_k\subset C([0,T], E_\xi)$  defined as
\[
v_k(t):=t^\mu F(u_k)(t)\,,\qquad t\in[0,T]\,, \quad k\in\N\,.
\] 
Note that $v_k(0)=0$ for $k\in\N$.
Since $(v_k(t))_k$ is bounded in $E_\eta$  for~$\eta\in(\xi,\mu]$ by~\eqref{eqeta}, 
the compactness of the embedding $E_\eta\hookrightarrow E_\xi$ ensures that $(v_k(t))_k$ is  relatively compact in $E_\xi$ for each~${t\in[0,T]}$. 

We next prove that the sequence~$(v_k)_k$ is equi-continuous.
Given~$0< s<t\le T$ and $k\in\N$, we have
\begin{subequations}\label{gfg}
\begin{equation}\label{gfg0}
\begin{aligned}
\|v_k(t)-v_k(s)\|_\xi&\le (t^\mu-s^\mu)\|F(u_k)(t)\|_\xi
+s^\mu\int_s^t \|e^{(t-\tau)A}\|_{\mathcal{L}(E_\gamma,E_\xi)}
 \,\|{ f(\tau,u_k(\tau))}\|_\gamma\,\rd \tau\\
&\quad +s^\mu\int_0^s \|e^{(t-\tau)A}-e^{(s-\tau)A}\|_{\mathcal{L}(E_\gamma,E_\xi)}
 \,\|{ f(\tau,u_k(\tau))}\|_\gamma\,\rd \tau \\	
 &\quad +s^\mu \, \|(e^{tA}-e^{sA})u^0\|_{\xi}=:I_1+I_2+I_3+ I_4\,.
\end{aligned}
\end{equation}
 In the following, we simply write $c$ for positive constants $c=c( T,L)$ depending on $T$ and $L$.
From~\eqref{o1}, we obtain 
$$
\|F(u_k)(t)\|_\xi\leq  c t^{-\xi}\,,\qquad t\in(0,T]\,,\quad k\in\N\,,
$$ 
and therefore
\begin{align}\label{gfg1}
I_1 \le ct^{-\xi}(t^\mu-s^\mu)\le ct^{-\xi}(t -s )^\mu \leq  c(t-s)^{\mu-\xi}\,.
\end{align}
 Using \eqref{e6} and~\eqref{e8}, we deduce 
\begin{align}
I_2 &\le    cs^\mu\int_s^t (t-\tau)^{\gamma_0-\xi}  \,\tau^{-\mu}\,\rd \tau \nonumber\\
&\leq cs^\mu(t-s)^{1+\gamma_0-\xi}\int_0^1 (1-r)^{\gamma_0-\xi}(r(t-s)+s)^{ -\mu}\,\rd\tau \nonumber\\
&\le   c(t-s)^{1+\gamma_0-\xi}\,.\label{gfg3}
\end{align}
Let  now $\eta\in  (\xi, \mu)$ be fixed  and $\gamma_0=\gamma_0(\eta)$ as in \eqref{gamma_0}.
Since 
\begin{align}\label{star}
\|e^{(t-s)A}-1\|_{\mathcal{L}(E_\eta,E_\xi)}\le c(t-s)^{\eta-\xi}
\end{align}
due to  \cite[II.~Theorem~5.3.1]{LQPP} (with $f=0$ therein),
 we use \eqref{e6} and \eqref{e8} to derive
\begin{align}
I_3 &\le s^\mu\int_0^s \|e^{(t-s)A}-1\|_{\mathcal{L}(E_\eta,E_\xi)} \,\|e^{(s-\tau)A}\|_{\mathcal{L}(E_\gamma,E_\eta)} \|f(\tau,u_k(\tau))\|_\gamma\,\rd \tau \nonumber\\
&\le c s^\mu(t-s)^{\eta-\xi}\int_0^s (s-\tau)^{\gamma_0-\eta}\, \tau ^{-\mu}\,\rd \tau \nonumber\\
&\le  c s^{1+\gamma_0-\eta}{\sf B}(1+\gamma_0-\eta, 1-\mu)\, (t-s)^{\eta-\xi}  \nonumber\\
&\le  c\, (t-s)^{\eta-\xi}\,. \label{gfg2}
\end{align}
Finally,  from ~\eqref{star} and~\eqref{e6} we get
\begin{align}
I_4 &\le s^\mu \|e^{(t-s)A}-1\|_{\mathcal{L}(E_\eta,E_\xi)} \,\|e^{ s A}\|_{\mathcal{L}(E_0,E_\eta)} \|u^0\|_0 \nonumber\\
&\le c T^{\mu-\eta}(t-s)^{\eta-\xi}\,\|u^0\|_0 \,. \label{gfg4}
\end{align}
\end{subequations}
We thus conclude  from~\eqref{gfg} that the sequence~$(v_k)_k$ is equi-continuous.
The Arzel\`a-Ascoli theorem  now ensures (up to a subsequence) that $v_k\to v$ 
in~${C([0,T],E_\xi)}$ for some function $v$ with $v(0)=0.$
Defining $u(t):=t^{-\mu}v(t),$ $t\in(0,T]$, we obtain~${F(u_k)\to u}$ in~$X_T$.
This establishes the compactness of $F$.\medskip

\noindent{\bf A priori bound.} We prove that the set 
$$
S:=\big\{u\in X_T\,:\, \text{$u=\lambda F(u)$ for some $\lambda\in[0,1]$}\big\}
$$
is bounded in $X_T$. To this end we infer from  \eqref{a1B1} and \eqref{e6} that   for $u\in S$ and $t\in (0,T]$ we have
\begin{equation*}
\begin{aligned}
\|u(t)\|_\xi&\leq\|F(u)(t)\|_\xi \le \|e^{tA}\|_{\mathcal{L}(E_0,E_\xi)}\, \|u^0\|_{0}
+\int_0^t \|e^{(t-\tau)A}\|_{\mathcal{L}(E_\gamma,E_\xi)} \,\|f(\tau,u(\tau))\|_\gamma\,\rd \tau\\
&\le   c  t^{-\xi} + c \int_0^t (t-\tau)^{\gamma_0-\xi}\big(1+\|u(\tau)\|_\xi\big)\,\rd \tau\\
&\le   c  t^{-\mu} + c \int_0^t (t-\tau)^{\gamma_0-\xi}\|u(\tau)\|_\xi\,\rd \tau\,,
\end{aligned}
\end{equation*}
with $\gamma_0=\gamma_0(\xi)$ defined as in \eqref{gamma_0}.
The singular Gronwall inequality \cite[II.Corollary 3.3.2]{LQPP} now implies the existence of a constant $K=K(T)>0$
such that 
\[
\|u(t)\|_\xi\leq K t^{-\mu}\,,\qquad t\in (0,T]\,,
\] 
which entails the boundedness of $S$ in $X_T$.\medskip

\noindent{\bf Existence of a solution.}
Due to the above considerations, we may apply Leray-Schauder's fixed point theorem  \cite[Theorem 11.3]{Gilbarg2001} 
to deduce that there exists $u\in X_T$ such that $F(u)=u$. 
It remains to prove that $u$ is a strong solution to the Cauchy problem \eqref{EE}.

To this end, we note that if $u_0\in E_\alpha$ and $\alpha\in[0, \min\{1,1+\gamma-\xi\})$, then~${u\in C([0,T], E_\alpha)}$.
 Indeed, the continuity of $u$ on $(0,T]$ is established in \eqref{Xregdes2}.
Since~${\big[t\mapsto e^{tA}u_0\big]\in C( [0,T],E_\alpha)}$, it follows from \eqref{eqeta} 
(with $u_0=0$, $\eta=\alpha$, and $\gamma_0=\gamma_0(\alpha)$ and $ \mu$ chosen such that  
$\alpha<1+\gamma_0-\mu$) that~${u=F(u)\in C([0,T], E_\alpha)}$, in particular $u(0)=u^0$.

Given $\ve\in (0,T)$, we  set as before $u_\ve(t):=u(t+\ve)$ for $t\in [0,T-\ve]$ and obtain from  the continuity of $f$  that $f(\cdot+\ve,u_\ve)\in C([0,T-\ve],E_\gamma)$.

 If~${\gamma>0}$, we may apply   \cite[II.Theorem~1.2.2, II.Remarks~2.1.2 (e)]{LQPP} to conclude from \eqref{P1d} that 
\[
u_\ve=F(u)(\cdot+ \ve) \in  C\big((0,T-\varepsilon],E_1\big)\cap C^1\big((0,T-\ve],E_0\big)\,,\quad \ve\in(0,T)\,,
\]
 hence 
\begin{equation}\label{regdes}
u\in C\big((0,T],E_1\big)\cap C^1\big((0,T],E_0\big)
\end{equation}
is a strong solution to \eqref{EE}.

If~$\gamma=0$,  we have  $u_\ve(0)\in E_\theta$ for some ~$\theta\in (\xi,1)$ by \eqref{Xregdes2} and $f(\cdot+\ve, u_\ve)\in C([0,T-\ve],E_0)$ so that \cite[II.Theorem~5.3.1]{LQPP}  together with~\eqref{P1d} 
imply  $u_\ve\in C^{\theta-\xi}([0,T-\ve],E_\xi)$. 
Along with~\eqref{a1B2} and \eqref{a1B'} we conclude~$f(\cdot+\ve, u_\ve)\in C^{\rho}([0,T-\ve],E_0)$ with $\rho:=\min\{\vartheta_0,\vartheta(\theta-\xi)\}\in(0,1)$.
   Invoking now \cite[II.Theorem~1.2.1]{LQPP} we deduce from \eqref{P1d} 
    \[
    u_\ve\in   C\big((0,T-\varepsilon],E_1\big)\cap C^1\big((0,T-\ve],E_0\big)
    \]
for each $\ve\in (0,T)$, hence $u$ is a strong solution to \eqref{EE} enjoying  the regularity properties \eqref{regdes}. This proves Theorem~\ref{MT1}.
 \qed

\subsection{Proof of Proposition~\ref{MT0}} 

We now provide the proof of Proposition~\ref{MT0}, which relies on the same strategy as that of Theorem~\ref{MT1}, the main difference being that the Leray-Schauder fixed point theorem
 is now applied in the Banach space $C([0,T],E_\alpha)$ of continuous  functions in $E_\alpha$.

Assume the premises of Proposition~\ref{MT0}. Set $X_T:=C([0,T],E_\alpha)$ and let $u\in X_T$.
 Since $\alpha\geq\xi$, the continuity of $f$ ensures~$f(\cdot,u)\in C([0,T],E_\gamma)$, so that $F(u)$ in \eqref{fpe} satisfies 
\begin{equation*}
F(u)\in C^{\alpha-\beta}\big([0,T],E_\beta\big)\cap C\big((0,T],E_\theta\big)\,,\qquad \beta\in [0,\alpha]\,,\quad \theta\in (0,1)\,,
\end{equation*}
due to  \cite[II.Theorem~5.3.1]{LQPP}  and~\eqref{Xregdes2}.
In particular,  $F(u)$ is an element of $X_T$.
Moreover, a straightforward contradiction argument shows that given $\ve>0$ arbitrary, there is $\delta>0$ such that
\begin{equation*}
\|f(t,u(t))-f(t,v(t))\|_{\gamma}\le  \ve\,,\qquad t\in [0,T]\,,\quad v\in \bar{\mathbb{B}}_{X_T}(u,\delta)\,.
\end{equation*}
From this we deduce, analogously to \eqref{c111}, that
\begin{align*}
\|F(u)(t)-F(v)(t)\|_\alpha&\le c(T) \ve\,,\qquad t\in [0,T]\,,\quad v\in \bar{\mathbb{B}}_{X_T}(u,\delta)\,,
\end{align*}
and thus the continuity of $F:X_T\to X_T$. 

To establish the  compactness of $F$, let ${(u_k)_k}$ be a sequence  in $X_T$ with $\|u_k\|_{X_T}\le L$ for some~${L>0}$. 
Then, in view of~\eqref{w1},  \eqref{gfg} reduces to
\begin{align*}
\|F(u_k)(t)-F(u_k)(s)\|_\alpha&\le \int_s^t \|e^{(t-\tau)A}\|_{\mathcal{L}(E_\gamma,E_\alpha)}
 \,\|{ f(\tau,u_k(\tau))}\|_\gamma\,\rd \tau\\
&\quad   +\int_0^s \|e^{(t-\tau)A}-e^{(s-\tau)A}\|_{\mathcal{L}(E_\gamma,E_\alpha)}
 \,\|{ f(\tau,u_k(\tau))}\|_\gamma\,\rd \tau\\
 &\quad   +\|(e^{tA}-e^{sA})u^0\|_\alpha\\
& \le c(T,L)(t-s)^{\eta-\alpha}+\|(e^{tA}-e^{sA})u^0\|_\alpha,  \qquad k\in\N\,, \quad 0< s<t\le T\,, 
\end{align*}
  where $\eta\in (\alpha,1)$ and $[t\mapsto e^{tA}u_0]\in C([0,T], E_\alpha)$.
  This establishes the equi-continuity of~${(F(u_k))_k}$.
  Moreover, arguing as in  \eqref{eqeta}, we also have (due to \eqref{w1}),  for some fixed $\eta\in (\alpha,1)$,
 \begin{equation*}
\|F(u_k)(t)\|_\eta \le   c(T,L)   t^{\alpha-\eta} ,  \qquad k\in\N\,, \quad 0<t\le T\,.
\end{equation*} 
  The Arzel\`a-Ascoli theorem now  guarantees, in view of the compactness of the embedding $E_\eta\hookrightarrow E_\alpha$, that~${(F(u_k))_k}$ is relatively compact in $X_T$. 
  
   Finally, the proof of an a priori bound for the set 
$$
S:=\big\{u\in X_T\,:\, \text{$u=\lambda F(u)$ for some $\lambda\in[0,1]$}\big\}
$$ 
is the same as in  Theorem~\ref{MT1} (using $u^0\in E_\alpha$ and \eqref{w1}). 
   Thus, $F$ admits a fixed point $u\in X_T$, which is a strong solution to~\eqref{EE} by the same arguments as in the last part of the proof of Theorem~\ref{MT1},  taking directly $\ve=0$ there if $\gamma>0$, respectively using \eqref{w2MT0} in the case $\gamma=0$. 
    This proves Proposition~\ref{MT0}.
\qed\medskip

We conclude this section by providing the proof for Corollary~\ref{MT2}.

 \subsection{Proof of Corollary~\ref{MT2}}

We first prove the claim for $u^0\in E_\alpha$ with $\alpha\in [0,\min\{1,1~+~\gamma~-~\xi\})$.
 Since $f$ satisfies \eqref{ASS}, and if $\gamma=0$ also~\eqref{a1B'} for some~$T_0>0$, Theorem~\ref{MT1}  guarantees the existence of a strong solution
 \begin{equation*} 
\begin{aligned}
 u_1&\in C^1\big((0,T_0],E_0\big)\cap C\big((0,T_0],E_1\big)\cap  C \big([0,T_0],E_\alpha\big)\cap C_\mu\big((0,T_0],E_\xi\big)
\end{aligned}
\end{equation*}
for some  $\mu\in(\xi,1)$, see \eqref{much},   to
\begin{equation}\label{EE1}
u_1'=Au_1+f(t,u_1)\,,\quad t\in(0,T_0]\,,\qquad u_1(0)=u^0\,.
\end{equation}
Noticing that $u_1(T_0)\in E_1\hookrightarrow E_\xi$, we infer from  Proposition~\ref{MT0}  (see \eqref{linbound} and \eqref{w2} if $\gamma=0$) that there exists a strong solution
\begin{equation*} 
\begin{aligned}
 u_2&\in C^1\big((0,T_0],E_0\big)\cap C\big((0,T_0],E_1\big)\cap  C \big([0,T_0],E_\xi\big)
\end{aligned}
\end{equation*}
 to
\begin{equation}\label{EE2}
u_2'=Au_2+f(t+T_0,u_2)\,,\quad t\in(0,T_0]\,,\qquad u_2(0)=u_1(T_0)\,.
\end{equation}
 Setting
$$
u(t):=\left\{\begin{array}{ll}
u_1(t)\,, &0\le t\le T_0\,,\\
u_2(t-T_0)\,, &T_0\le t\le 2T_0\,,
\end{array}\right.
$$
we obtain $u\in  C_\mu\big((0,2T_0],E_\xi\big)$. In particular, $f(\cdot,u)\in  C \big((0,2T_0],E_\gamma\big)$ satisfies \eqref{e8} (with $T=2T_0$ therein). Since $u_1$ and $u_2$ are given by the corresponding mild formulation  of~\eqref{EE1} respectively~\eqref{EE2} (see~\eqref{ttp}), it readily follows that $u$ satisfies
\begin{equation*}
u(t)=e^{tA}u^0+ \int_0^t e^{(t-\tau)A} f(\tau,u(\tau))\,\rd \tau\,,\qquad t\in[0,2T_0]\,.
\end{equation*}
We may now argue as in the proof of Theorem~\ref{MT1} to deduce that
\begin{equation*} 
\begin{aligned}
 u&\in C^1\big((0,2T_0],E_0\big)\cap C\big((0,2T_0],E_1\big)\cap  C \big([0,2T_0],E_\alpha\big)
\end{aligned}
\end{equation*}
is a strong solution to~\eqref{EE} on $[0,2T_0]$. 
Consequently,  using Proposition~\ref{MT0}, we may inductively extend the solution to~$[0,nT_0]$ for $n\ge 2$ and obtain in this way a global strong solution
\begin{equation*} 
\begin{aligned}
 u&\in C^1\big((0,\infty),E_0\big)\cap C\big((0,\infty),E_1\big)\cap  C \big([0,\infty),E_\alpha\big)
\end{aligned}
\end{equation*}
to~\eqref{EE} on $[0,\infty)$. 

The assertion for $u^0\in E_\alpha$ with $\alpha\in [\xi,1)$ follows by applying  Proposition~\ref{MT0} successively on~${[n,n+1]}$ for $n\in\N$   and gluing the solutions as above.  
This proves Corollary~\ref{MT2}.
\qed

\section{Applications}\label{Sec:3}
We apply the abstract theory to  concrete models.
 As a starting point we provide in Section~\ref{Sec:31} a  functional analytic framework for the  subsequent applications.
 In Section~\ref{Sec:32} we then analyze  an evolution equation for the spreading of bushfires 
 and  in Section~\ref{Sec:33}  we apply the abstract results to a reaction-diffusion system.
 The examples presented in Section~\ref{Sec:32}--Section~\ref{Sec:33}
  effectively illustrate the significance of both Proposition~\ref{MT0} and Theorem~\ref{MT1} in practical applications.

\subsection{Functional Analytic Setting}\label{Sec:31}

Let   $\Omega\subset\R^n$, $n\in\N^*$,  be an open and bounded set with  smooth boundary and outward unit normal $\nu$.
In order to include Dirichlet and Neumann boundary conditions, we choose~${\delta\in\{0,1\}}$ and define
$$
\mathcal{B}u:=u  \ \text{ on } \ \partial\Omega \ \text{ if } \ \delta=0\,,\qquad \mathcal{B}u:=\partial_\nu u\ \text{ on } \ \partial\Omega \ \text{ if } \ \delta=1\,,
$$
hence $\mathcal{B}$ is the trace operator if $\delta=0$  corresponding to Dirichlet boundary conditions, while $\delta=1$ refers to Neumann boundary conditions.
 We introduce
$$
F_0:=L_2(\Omega)\,,\qquad F_1:= H_{\mathcal{B}}^{2}(\Omega)=\{v\in H^{2}(\Omega)\,:\,  \mathcal{B} v=0 \text{ on } \partial\Omega\}\,, 
$$
with $H^{2}(\Omega)=W_{2,\mathcal{B}}^{2}(\Omega)$ denoting the Sobolev space of second-order over $L_2(\Omega)$,
and recall that
$$
B_0:=\Delta_\mathcal{B}\in \mathcal{H}\big(H_\mathcal{B}^{2}(\Omega),L_2(\Omega)\big)\,,
$$
 that is, $B_0=\Delta_\mathcal{B}$ with domain $ H_\mathcal{B}^{2}(\Omega)$ generates an analytic semigroup on $F_0=L_2(\Omega)$.
Let
$$
\big\{(F_\theta,B_\theta)\,:\, -1\le \theta<\infty\big\}\ 
$$
be the interpolation-extrapolation scale generated by $(F_0,B_0)$ and the 
complex interpolation functor~$[\cdot,\cdot]_\theta$ (see \cite[\S 6]{Amann_Teubner} and \cite[\S V.1]{LQPP}), that is,
\begin{equation}\label{f2x}
B_\theta\in \mathcal{H}(F_{1+\theta},F_\theta)\,,\quad -1\le \theta<\infty\,,
\end{equation}
  and, for $2\theta\neq -\delta-1/2$, we have  (see \cite[Theorem~7.1;   Equations (7.4)-(7.5)]{Amann_Teubner})
\begin{equation}\label{f2}
 F_\theta\doteq H_{\mathcal{B}}^{2\theta}(\Omega):=\left\{\begin{array}{ll} \{v\in H^{2\theta}(\Omega) \,:\, \mathcal{B} v=0 \text{ on } 
 \partial\Omega\}\,, &\delta+\frac{1}{2}<2\theta\leq 2 \,,\\[3pt]
	 H^{2\theta}(\Omega)\,, &  -\frac{3}{2}+\delta< 2\theta<\delta +\frac{1}{2}\,.\end{array} \right.
\end{equation}
Since $\Delta_\mathcal{B}$ has bounded imaginary powers, \cite[Remarks~6.1~(d)]{Amann_Teubner}  provides  the following reiteration property 
\begin{equation}\label{f3}
 [F_\alpha,F_\beta]_\theta\doteq F_{(1-\theta)\alpha+\theta\beta}.
\end{equation}

\subsection{A model for front propagation in bushfires}\label{Sec:32}

In \cite{DVWW24ax, DVWW24bx} the authors recently proposed and studied, both from an analytical and numerical point of view, the nonlocal equation 
\begin{subequations}\label{Bush}
\begin{equation}\label{Bush1}
\p_t u=\Delta u+\int_{\Omega}\big (u(t,y)-\Theta(t,y)\big)_+K(x,y)\,{\rm d}y+\Big(\Big(\omega+\frac{\beta (u)\nabla u}{|\nabla u|^\nu}\Big)\cdot\nabla u\Big)_-,\qquad x\in\Omega\,,\quad t>0\,,
\end{equation}
describing front propagation in bushfires, where the unknown $u=u(t,x)$ is  the  environmental temperature. 
 Here,~$r_\pm:=\max\{0,\pm r\}$ for~${r\in\R}$ and  $\nu\in[1,2]$.
We assume that $\Omega$ is  on an open, bounded, smooth subset  of $\R^n$  with $n\geq2$ (for an extension of the subsequent results to Lipschitz domains $\Omega$ one may use \cite{Wood}). 
We refer to \cite{DVWW24ax, DVWW24bx} for the physical interpretation of the integral kernel~$K$  and the functions $\Theta,\, \omega$, and $\beta$.
Equation \eqref{Bush1} is supplemented by the Dirichlet boundary condition
\begin{equation}\label{Bush2}
u=0\ \text{ on $\p\Omega$}\,,\qquad t>0\,, 
\end{equation}
 and is subject to the initial condition
\begin{equation}\label{Bush3}
u(0,x)=u^0(x)\,,\qquad x\in\Omega\,. 
\end{equation}
\end{subequations}
In \cite{DVWW24bx} the authors established the existence of local and global  weak solutions to \eqref{Bush} for initial data~$u^0\in H^1(\Omega)$ with $u=0$ on $\partial\Omega$. 
The following theorem provides, under slightly more restrictive regularity assumptions on $\Theta $ and~$\omega$, but less restrictive integrability assumptions on $\Theta$ 
(and no size conditions on $\beta$, $\omega$, and $\Theta$  when $\nu=1$)  compared to \cite{DVWW24bx}, 
global strong solutions to~\eqref{Bush}  for the larger class of initial data~${u^0\in L_2(\Omega)}$. 

In the context of \eqref{Bush} we define 
\[
 H_D^{2\theta}(\Omega):=H^{2\theta}_{\mathcal{B}}(\Omega)\,, \qquad 2\theta\in\Big(-\frac{3}{2}, 2\Big]\setminus \Big\{\pm \frac{1}{2}\Big\}\,,
\] 
 with Dirichlet boundary conditions $\mathcal{B}u=0$ on $\partial\Omega$, see~\eqref{f2}  with $\delta=0$.
With this notation our result reads as follows.

 \begin{thm}\label{T:A1} Let~$\nu\in[1,2]$  and assume that 
$$
K\in L_2(\Omega\times\Omega)\,,\quad \beta\in W^1_\infty(\R)\,,\quad  \Theta\in C( [0,\infty), L_2(\Omega))\,,\quad \omega \in C([0,\infty), L_\infty(\Omega, \R^n))\,,
$$ 
and~$2\ve\in(0,1/7)$.
 Then, given $u^0\in L_2(\Omega)$, there exists a global  strong solution to \eqref{Bush} such that
 \begin{equation}\label{regsol}
  u\in C^1\big((0,\infty),H^{-2\ve}(\Omega)\big)\cap C\big((0,\infty),H^{2-2\ve}_D(\Omega)\big)\cap  C \big([0,\infty),L_2(\Omega)\big)\,.
  \end{equation}
 
Moreover,
 \begin{itemize}
 \item[(i)] if $\nu=2$, then the solution is unique; 
  \item[(ii)]    if, in addition,   $\Theta\in C^\vartheta([0,\infty), L_2(\Omega))$ and   $ \omega \in C^\vartheta( [0,\infty), L_\infty(\Omega, \R^n)) $ for some $\vartheta\in (0,1)$, then the solution  satisfies
   \[
  u\in C^1\big((0,\infty),L_2(\Omega)\big)\cap C\big((0,\infty),H^{2}_D(\Omega)\big)\,.
  \]
 \end{itemize}
 \end{thm}
 
 Concerning the uniqueness statement for $\nu=2$ we add the following observation.
 
 \begin{rem}
 The solution found in Theorem~\ref{T:A1} satisfies additionally $u\in C_\mu((0,T], H^{1+2\ve}_D(\Omega))$ for  some fixed $\mu\in (1/2+2\ve,1)$  and each $T>0$. 
For $\nu=2$, the uniqueness of the solution,  for functions which satisfy \eqref{regsol} and 
belong to $  C_\mu((0,T], H^{1+2\ve}_D(\Omega))$ for each $T>0$, can be also shown by 
 using directly  the fixed point formulation \eqref{ttp}  together with the singular Gronwall inequality \cite[II.Corollary 3.3.2]{LQPP} 
 (since the nonlinearity $f=f(t,u)$ is locally Lipschitz continuous in $u$,  see~\eqref{A} and  Lemma~\ref{L:A1} below). 
 However, the uniqueness result in Theorem~\ref{T:A1}~(i) is more general. 
  \end{rem}
  
Before proving Theorem~\ref{T:A1} we  note some auxiliary results. 
The following lemma is related to \cite[Lemma~2.2, Lemma~2.4]{DVWW24ax}.

\begin{lem}\label{L:A1} Let $K\in L_2(\Omega\times\Omega).$ 
Then the mapping $g_0:L_2(\Omega)^2\to L_2(\Omega)$, given by
\[
g_0(u,\Theta)(x):=\int_{\Omega}\big (u(y)-\Theta(y)\big)_+K(x,y)\,{\rm d}y\,,\qquad x\in\Omega\,,\quad u, \Theta\in L_2(\Omega)\,,
\]
is well-defined and satisfies
\begin{align*}
&\|g_0(u,\Theta)\|_2\leq \|K\|_2\big(\|\Theta\|_2+\|u\|_2\big),\qquad u, \Theta\in L_2(\Omega)\,,
\end{align*}
and
\begin{align*}
&\|g_0(u_1,\Theta_1)-g_0(u_2,\Theta_2)\|_2\leq \|K\|_2\big(\|u_1-u_2\|_2+\|\Theta_1-\Theta_2\|_2\big)\,,\qquad u_i, \Theta_i\in L_2(\Omega),\quad i=1\,,2\,.
\end{align*}
\end{lem}
\begin{proof}
The claims follow by using H\"older's inequality together with the inequality $|a_+-b_+|\leq |a-b|$ for $a,\, b\in\R.$
\end{proof}

 We turn to the second, nonlinear term  on the right of \eqref{Bush1}.
To this end we define  for $\nu\in[1,2]$ the mapping 
$$
g_\nu:L_2(\Omega)\times L_2(\Omega, \R^n)\times L_\infty(\Omega, \R^n) \to L_2(\Omega)
$$
according to
\[
g_\nu(u,p,\omega):=\big(\omega\cdot p +\beta (u)|p|^{2-\nu}\big)_-.
\]
 The following result was observed in \cite[Lemma~2.1, Lemma~2.3, Lemma~2.5]{DVWW24ax} (the proof  therein is valid also for $\nu=1)$.

\begin{lem}\label{L:A2} If $\beta\in W^1_\infty(\R)$ and $\nu\in[1,2]$,  then for  $(u,p, \omega)\in  L_2(\Omega)\times L_2(\Omega,\R^n)\times L_\infty(\Omega,\R^n)$ we have
\begin{align*}
&\|g_\nu(u,p,\omega)\|_2\leq \|\omega\|_\infty\|p\|_2+|\Omega|^{(\nu-1)/2}\|\beta\|_\infty\|p\|_2^{2-\nu}
\end{align*}
 and for  $(u_i,p_i, \omega_i)\in  L_2(\Omega)\times L_2(\Omega,\R^n)\times L_\infty(\Omega,\R^n)$, $i=1\,,2$,
\begin{align*}
\|g_\nu(u_1,p_1&,\omega_1)-g_\nu(u_2,p_2,\omega_2)\|_2\\
&\leq \|\omega_1\|_\infty\|p_1-p_2\|_2+\|p_2\|_2\|\omega_1-\omega_2\|_\infty+2\|\beta\|_{W^1_\infty}\|p_1\|_{2}^{ 2-\nu}\|u_1-u_2\|_2^{\nu-1}\\
&\quad+|\Omega|^{(\nu-1)/2}\|\beta\|_\infty\|p_1-p_2\|_2^{2-\nu}\,.
\end{align*}
 In fact, for $\nu=2$,
\begin{equation}
\begin{split}\label{A}
\|g_2(u_1&,p_1,\omega_1)-g_2(u_2,p_2,\omega_2)\|_2\\
&\leq \|\omega_1\|_\infty\|p_1-p_2\|_2+\|p_2\|_2\|\omega_1-\omega_2\|_\infty+\|\beta\|_{W^1_\infty}\|u_1-u_2\|_2\,.
\end{split}
\end{equation}
\end{lem}

\begin{proof}
The claims follow by using H\"older's inequality together with the inequality $|a_--b_-|\leq |a-b|$ for  $a,\, b\in\R $.
\end{proof}

We note   for $\nu=1$ that the estimate in Lemma~\ref{L:A2}  does not provide local H\"older continuity of~$g_1$. 
This is achieved  now in Lemma~\ref{L:A3} by restricting  the range for $p$.

\begin{lem}\label{L:A3}  If $\beta\in W^1_\infty(\R)$ and $2\ve\in(0,1/2)$, then there is $c>0$ with
\begin{align*}
\|g_1(u_1,p_1,\omega_1)-g_1(u_2,p_2,\omega_2)\|_2&\leq \big(\|\omega_1\|_\infty+\|\beta\|_\infty\big)\|p_1-p_2\|_2+\|p_2\|_2\,\|\omega_1-\omega_2\|_\infty\\
&\quad+ c\|\beta\|_{W^1_\infty}\,\|p_1\|_{H^{2\ve}}\, \|u_1-u_2\|_2^{4\ve/n}
\end{align*}
for all $ (u_i,p_i, \omega_i)\in  L_2(\Omega)\times H^{2\ve}(\Omega, \R^n)\times L_\infty(\Omega, \R^n)$, $i=1,\,2.$
\end{lem}
\begin{proof}
The details are identical to that in Lemma~\ref{L:A2}, except those used when estimating the term~${\|p_1(\beta(u_1)-\beta(u_2))\|_2}$.
Letting $r_\ve,\,q_\ve\in(2,\infty)$ be given by 
\[
\frac{1}{r_\ve}:=\frac{1}{2}-\frac{2\ve}{n},\qquad \frac{1}{q_\ve}:=\frac{1}{2}-\frac{1}{ r_\ve}=\frac{2\ve}{n},
\] 
H\"older's inequality together with the embedding $H^{2\ve}(\Omega)\hookrightarrow L_{ r_\ve}(\Omega)$  implies
\begin{align*}
\|p_1(\beta(u_1)-\beta(u_2))\|_2&\leq \|p_1\|_{r_\ve}\|\beta(u_1)-\beta(u_2)\|_{q_\ve}\leq  c\|p_1\|_{H^{2\ve}}\Big(\int_\Omega|\beta(u_1)-\beta(u_2)|^{q_\ve}\,{\rd}x\Big)^{1/q_\ve}\\
&\leq   c\|p_1\|_{H^{2\ve}}(\|\beta\|_\infty)^{\frac{q_\ve-2}{q_\ve}}\Big(\int_\Omega|\beta(u_1)-\beta(u_2)|^{2}\,{\rd}x\Big)^{1/q_\ve}\\
&\leq  c \|\beta\|_{W^1_\infty}\|p_1\|_{H^{2\ve}}\|u_1-u_2\|_2^{2/q_\ve},
\end{align*}
which provides the desired estimate.
\end{proof}

We are now in a position to establish Theorem~\ref{T:A1}.

\subsection*{Proof of Theorem~\ref{T:A1}}

Let   $2\ve\in(0, 1/7)$ be fixed and  set 
\[
 E_0:= H^{-2\ve}_D(\Omega)\,,\qquad  E_1:= H^{2-2\ve}_D(\Omega)\,.
\]
Defining  $2\xi:=1+ 4\ve$ and $\gamma:=\ve\in(0,1)$, we infer from \eqref{f3} that
 \[
E_\xi = H^{ 1+2\ve}_D(\Omega)\,,\qquad E_\gamma  = L_2(\Omega)\,.
\]
We may thus recast \eqref{Bush}  as a semilinear parabolic Cauchy problem
\begin{equation}\label{EP:A1}
u'=\Delta u+f_\nu(t,u)\,,\quad t>0\,,\qquad u(0)=u^0\,,
\end{equation}
where $\Delta\in \mathcal{H}(E_1,E_0)$  (with Dirichlet boundary conditions, see \eqref{f2x}-\eqref{f2}) and $f_\nu: [0,\infty)\times E_\xi\to E_\gamma$ 
for $\nu\in[1,2]$ is defined by 
\begin{equation*}
f_\nu(t,u):=g_0(u,\Theta(t))+g_\nu(u,\nabla u,\omega(t)),\qquad (t,u)\in [0,\infty)\times E_\xi.
\end{equation*}
Lemma~\ref{L:A1}--Lemma~\ref{L:A3}  ensure that $f_\nu\in C( [0,\infty)\times E_\xi, E_\gamma)$ has the property that for each $T>0$ there exists $C(T)>0$ such that 
\begin{equation*}
\|f_\nu(t,u)\|_{E_\gamma}\leq C(T)(1+\|u\|_{E_\xi})\,,\qquad t\in[0,T]\,,\quad u\in E_\xi\,,
\end{equation*}
and, if $\nu\in (1,2]$, then for all $t\in[0,T]$ and $u_1,\, u_2\in E_\xi$,
\begin{equation*}
\|f_\nu(t,u_1)-f_\nu(t,u_2)\|_{E_\gamma}\leq C(T)\big(\|u_1-u_2\|_{E_\xi}^{\vartheta_1}+\|u_1\|_{E_\xi}^{q_2-\vartheta_2}\|u_1-u_2\|_{E_\xi}^{\vartheta_2}+\|u_1-u_2\|_{E_\xi}^{\vartheta_3}\big)
\end{equation*}
with 
$$
\vartheta_1:=1=:q_1\,,\quad 0\leq \vartheta_2:=\nu-1\leq q_2:=1 \,,\quad 0\leq \vartheta_3:=2-\nu=:q_3\leq 1\,,
$$ 
while,
for $\nu=1$, we have
\begin{equation*}
\|f_1(t,u_1)-f_1(t,u_2)\|_{E_\gamma}\leq C(T)\big(\|u_1-u_2\|_{E_\xi}^{\vartheta_1}+\|u_1\|_{E_\xi}^{q_1-\vartheta_1}\|u_1-u_2\|_{E_\xi}^{\vartheta_2} \big)
\end{equation*}
with 
$$\vartheta_1:=1=:q_1\,,\quad 0\leq \vartheta_2:= 4\ve/n\leq q_2:=1+\vartheta_2\,.
$$
The largest value $q$ of the exponents $q_j$ is easily computed as
\[
q:=
\left\{
\begin{array}{ll}
1\,,&\nu\in(1,2]\,,\\[1ex]
1+4\ve/n\,,&\nu=1\,,
\end{array}
\right.
\]
and, since $n\geq 2$ and $2 \ve< 1/7$,  it holds that $\xi<1/q=\min\big\{1,1/q\big\}$.\\

This shows that  assumption~\eqref{ASS} is fulfilled in the context of \eqref{EP:A1} for each $T>0$.
We may  thus apply Corollary~\ref{MT2}~(a) with 
 $\alpha:=\ve\in [0,1/2-\ve)=[0,\min\{1,1+\gamma-\xi\})$   and deduce, for each~${u^0\in E_\alpha =L_2(\Omega)}$, the existence of  a global strong solution  
\begin{equation*} 
 u\in C^1\big((0,\infty),E_0\big)\cap C\big((0,\infty),E_1\big)\cap  C \big([0,\infty),E_\alpha\big) 
\end{equation*}
to ~\eqref{EP:A1}.\\

Regarding the uniqueness claim for $\nu=2$ stated in~(i), let    $u_1,\, u_2$ be  solutions to \eqref{Bush}  as found in Theorem~\ref{T:A1} such that $u_1(0)=u_2(0)$ and choose an arbitrary $T>0$. 
Using the Lions-Magenes  lemma, see \cite[Theorem~II.5.12]{BF13}, 
we  have for $0<\tilde t< t\leq T$ 
\begin{align*}
\|(u_1 -u_2)(t)\|_2^2-\|(u_1 -u_2)(\tilde t)\|_2^2&=-2\int_{\tilde t}^t\|\nabla (u_1-u_2)(\tau)\|_2^2\,{\rm d}\tau\\[1ex]
&\quad +2\int_{\tilde t}^t\int_\Omega  (u_1-u_2)(\tau) \big[f_2(\tau ,u_1(\tau))-f_2(\tau ,u_2(\tau))\big]\,{\rm d}x\,{\rm d}\tau\,.
\end{align*}
Since $\|f_2(\tau ,u_1(\tau))-f_2(\tau ,u_2(\tau))\|_2\leq C(T)\|(u_1-u_2)(\tau)\|_{H^1}$  for $ \tau\in(0, T]$ due to Lemma~\ref{L:A1} and~Lemma~\ref{L:A2}, 
 H\"older's inequality, Young's inequality, and the previous estimate lead, after taking the limit~$\tilde t\to0$, to
\begin{align*}
\|(u_1 -u_2)(t)\|_2^2\leq C(T)\int_0^t\|(u_1 -u_2)(\tau)\|_2^2\,{\rm d}\tau\,,\qquad 0\leq t\leq T\,.
\end{align*}
The desired assertion  follows now  from Gronwall's lemma.
\\

 Finally, let  $\Theta\in C^\vartheta([0,\infty), L_2(\Omega))$ and   $ \omega \in C^\vartheta([0,\infty), L_\infty(\Omega,\R^n)) $ for some $\vartheta\in (0,1)$ and fix~${\delta\in(0,1)}$.
Since $f(\cdot,u)\in C([\delta,\infty), L_2(\Omega))$, $u(\delta)\in H^{3/2}_0(\Omega)$,  and $\Delta\in \mathcal{H}(H^2_0(\Omega),L_2(\Omega))$, we infer from  \cite[II.Theorem~5.3.1]{LQPP}  that 
$u\in C^{\rho}([\delta,\infty), E_\xi)$ for some $\rho\in(0,1)$.
This property together with Lemma~\ref{L:A1}-Lemma~\ref{L:A3} show that $f(\cdot,u)\in C^{\rho'}([\delta,\infty), L_2(\Omega))$ for some ~${\rho'\in(0,\rho)}$.
 In view of \cite[II.Theorem~1.2.1]{LQPP}  we get $u\in C^1\big((\delta,\infty),L_2(\Omega)\big)\cap C\big((\delta,\infty),H^{2}_0(\Omega)\big)$ 
  and, since $\delta\in(0,1)$ is chosen arbitrarily, this establishes~(ii), 
and  the proof of Theorem~\ref{T:A1} is complete.\qed

 \subsection{Dynamics of an isothermal, autocatalytic chemical reaction scheme}\label{Sec:33}

 We consider a mathematical model  for the dynamics of an isothermal, autocatalytic chemical
reaction scheme with termination, taking place in an unstirred environment and undergoing
molecular diffusion  investigated in a one-dimensional setting in \cite{Meyer}.
 The governing equations are summarized by the system
 \begin{subequations}\label{acs}
\begin{align}
\partial_t u-\Delta u &=- u_+^\mu  v_+^\beta\,,&&\hspace{-1cm}t>0\,,\quad x\in\Omega\,,\label{ac1}\\
\partial_t v-\Delta v &= u_+^\mu v_+^\beta-a  v_+^\theta\,,&& \hspace{-1cm} t>0\,,\quad x\in\Omega\,,\label{ac2}
\end{align}
where $u$ and $v$ are concentrations of the  
reactant   and the autocatalyst, respectively,~$a $ is a (positive) constant, and the constants $\mu,\,\beta,\,\theta\in (0,1]$ are assumed to satisfy $\mu+\beta \le 1$.
As before, we set~${r_+:=\max\{0, r\}}$ for~${r\in\R}$.
Moreover,  $\Omega\subset\R^n$, $n\in\N^*$,  is an open and bounded set with  smooth boundary  and outward unit normal $\nu$.

We supplemented \eqref{ac1}-\eqref{ac2} by homogeneous Neumann boundary conditions
\begin{equation}\label{ac3}
\p_\nu u=\p_\nu v=0\qquad \text{on $\p\Omega$}\,,\quad t>0\,, 
\end{equation}
 and impose the initial condition
\begin{equation}\label{ac4}
u(0,x)=u^0(x)\,,\quad v(0,x)=v^0(x)\,,\qquad x\in\Omega\,. 
\end{equation}
\end{subequations}

In the context of the evolution problem~\eqref{acs} we chose the Hilbert spaces $H^{2\theta}_N(\Omega)$, which are defined by \eqref{f2}  with~${\mathcal{B}u=\p_\nu u}$, to establish the following result.

\begin{thm}\label{T:A2} Let~$a \in\R$ and  $\mu,\,\beta,\,\theta\in (0,1]$ satisfy $\mu+\beta\leq 1.$
 Then, given $u^0\in L_2(\Omega)$, there exists a global  strong solution to \eqref{acs} such that
 \begin{equation}\label{regsol2}
  (u,v)\in C^1\big((0,\infty),L_2(\Omega,\R^2)\big)\cap C\big((0,\infty),H^{2}_N(\Omega, \R^2)\big)\cap  C \big([0,\infty),L_2(\Omega, \R^2)\big)\,.
  \end{equation}
 \end{thm}
\begin{proof}
Set 
\[
E_0:=L_2(\Omega, \R^2)\,,\qquad E_1:=H^2_N(\Omega,\R^2)\,,
\]
 and $\xi:=\gamma:=0$.
The evolution problem~\eqref{acs} may be recast 
as the semilinear parabolic Cauchy problem
\begin{equation}\label{EP:A2}
w'=A  w+f(w)\,,\quad t>0\,,\qquad w(0)=w^0:=(u^0,v^0)\,,
\end{equation}
where  
$$
A:=\begin{pmatrix}
\Delta&0\\[1ex]
0&\Delta
\end{pmatrix}\in \mathcal{H}(E_1,E_0)\,,
$$
see  \cite[I.~Theorem 1.6.1]{LQPP}  and \eqref{f2x}-\eqref{f2}, and $f:E_0\to E_0$ is given by   
\[
f(w):=(- u_+^\mu  v_+^\beta,u_+^\mu  v_+^\beta-a  v_+^\theta),\qquad w=(u,v)\in E_0\,.
\]
In fact, using H\"older's inequality, for $w=(u,v)\in E_0$ we have
\begin{equation}\label{estts}
\begin{aligned}
&\|v_+^\theta\|_2\leq |\Omega|^{(1-\theta)/2}\|v\|_2^\theta\leq |\Omega|^{(1-\theta)/2}\|w\|_{E_0}^\theta\,,\\
&\|u_+^\mu  v_+^\beta\|_2\leq |\Omega|^{(1-\beta-\mu)/2}\|u\|_2^\mu\|v\|_2^\beta\leq |\Omega|^{(1-\beta-\mu)/2}\|w\|_{E_0}^{\mu+\beta}\,,
\end{aligned}
\end{equation}
which shows that $f$ is well-defined and satisfies
\begin{equation}\label{LG}
\|f(w)\|_{E_0}\le C\big(1+\|w\|_{E_0}\big),\qquad w\in E_0\,. 
\end{equation}
Moreover, \eqref{estts} together with the inequality 
$$
|a_+^\vartheta-b_+^\vartheta|\leq |a-b|^\vartheta\,,\qquad a,\, b\in\R\,,\quad \vartheta\in(0,1]\,,
$$
shows that there exists a constant $C>0$ such  that for all $w,\, \tilde w\in E_0$ we have
\begin{equation}\label{GL}
\|f(w)-f(\tilde w)\|_{E_0}\leq C(\|w-\tilde w\|_{E_0}^\theta+\|w\|_{E_0}^\mu\|w-\tilde w\|_{E_0}^\beta+\|w\|_{E_0}^\beta\|w-\tilde w\|_{E_0}^\mu)\,.
\end{equation}
In view of \eqref{LG} and \eqref{GL}, we may thus apply Corollary~\ref{MT2}~(b) (with $\alpha=\xi =\gamma=0$) in the context of  the evolution problem~\eqref{EP:A2} to deduce, for each~$w^0\in E_0$,
 the existence of a strong solution to~\eqref{acs} with the required regularity 
property~\eqref{regsol2}.
\end{proof}

There are other examples to which the abstract results Theorem~\ref{MT1} and Corollary~\ref{MT2} apply, e.g., to reaction–diffusion equations of the form
$$
\partial_t u-\Delta u =u\vert u\vert^{p-1}\,, \qquad t>0\,,\quad x\in\Omega\,,
$$
with $p\in (0,1)$,  see \cite{ClarkMeyer} and the references therein. 
In this case, the function   $f:L_2(\Omega)\to L_2(\Omega)$ with~${f(u):=u\vert u\vert^{p-1}}$ is linearly bounded
 and uniformly $p$--H\"older continuous.

\bibliographystyle{siam}
\bibliography{Literature}
\end{document}